\documentclass[reqno,12pt]{amsart}

\usepackage{color}
\usepackage{times}
\usepackage{amssymb, amsmath}
\allowdisplaybreaks

\theoremstyle{plain}
\newtheorem{thm}{Theorem}[section]
\newtheorem{cor}[thm]{Corollary}
\newtheorem{lem}[thm]{Lemma}

\theoremstyle{definition}

\theoremstyle{remark}
\newtheorem{rem}{Remark}[section]

\numberwithin{equation}{section}

\def\disp{\displaystyle}

\def\disp{\displaystyle }
\def\d{\displaystyle \frac}
\def\p{\partial}

\def\t1{\partial}

\def\t{\tilde}

\title[Degenerate lake equations]
{Vanishing viscosity limits for the degenerate lake equations with Navier boundary conditions}

\author[Quansen Jiu, Dongjuan Niu and Jiahong Wu]{}

\begin{document}
\maketitle

\centerline{\scshape Quansen Jiu }

\medskip
{\footnotesize
  \centerline{School of Mathematical Sciences, Capital Normal University}
  \centerline{Beijing  100048, P. R. China}
   \centerline{  \it Email: jiuqs@mail.cnu.edu.cn}
   }

   \vspace{3mm}
   \centerline{\scshape Dongjuan Niu }

\medskip
{\footnotesize
  \centerline{School of Mathematical Sciences, Capital Normal University}
  \centerline{Beijing  100048, P. R. China}
   \centerline{  \it Email: djniu@mail.cnu.edu.cn}
   }
   \vspace{3mm}
   \centerline{\scshape Jiahong Wu}

\medskip
{\footnotesize
  \centerline{Department of Mathematics, Oklahoma State University}
  \centerline{Stillwater, OK 74078, USA}
   \centerline{  \it Email: jiahong@math.okstate.edu}
}

\begin{abstract}
The paper is concerned with the vanishing viscosity limit of the
two-dimensional degenerate viscous lake equations when the Navier
slip conditions are prescribed on the impermeable boundary of a
simply connected bounded regular domain. When the initial vorticity
is in the Lebesgue space $L^q$ with $2<q\le\infty$, we show the
degenerate viscous lake equations possess a unique global solution
and the solution converges to a corresponding weak solution of the
inviscid lake equations. In the special case when the vorticity is
in $L^\infty$, an explicit convergence rate is obtained.


\end{abstract}

\section{Introduction}

Let $\Omega\subset \mathbb{R}^2$ be a simply connected bounded domain with a smooth boundary and let $\overline{\Omega}$
and $\partial\Omega$ denote its closure and boundary, respectively.
Let $I$ denote the $2\times 2$ identity matrix. Let $b(x)\in C^2(\overline{\Omega})$ be a given function with $b(x) >0$ for any $x\in \Omega$. We are not assuming that $b$ is nondegenerate, namely
that $b$ may be zero on $\partial\Omega$.
Consider the viscous lake equations
\begin{equation}\label{a1}
\left\{
\begin{array}{l}
\p_t u^{\mu} +u^{\mu}\cdot \nabla u^{\mu}-\mu
b^{-1}\nabla\cdot(2bD(u^{\mu})-b \nabla\cdot u^{\mu} I)+\nabla
p^{\mu}=0,\\[2mm]
\nabla\cdot(bu^{\mu})=0,
\end{array}
\right.
\end{equation}
where $x\in \Omega$, $t > 0$, $\mu>0$ represents the viscosity coefficient
and $u^{\mu}=u^{\mu}(x,t)$ stands for the two-dimensional velocity field
and $D(u^{\mu})$ the deformation tensor, namely
$$
D(u^{\mu})=\frac{\nabla
u^{\mu}+ (\nabla {u^{\mu}})^t}{2}.
$$
Attention here is focused on the initial- and boundary-value problem (IBVP) for \eqref{a1} with
the free boundary condition
\begin{equation}\label{a2}
b u^{\mu}\cdot n=0, \ \ \nabla\times u^{\mu}=0 ,\ \ \ x\in\p\Omega, \, t>0,
\end{equation}
and a given initial data
\begin{equation}\label{a3}
 u^{\mu}(x,t)\mid_{t=0}=u_0,\ \  x\in\Omega,
\end{equation}
where $n$ denotes the unit normal vector and $u_0$ is assumed to satisfy the boundary condition in
\eqref{a2} and $\nabla\cdot (b u_0)=0$. \eqref{a2} is a special case of the general Navier boundary
condition
\begin{equation}\label{navier}
b u^{\mu}\cdot n=0, \ \ 2D(u^\mu)n \cdot \tau + \alpha u\cdot \tau =0, \quad x\in\p\Omega
\end{equation}
and \eqref{navier} reduces to (\ref{a2}) when $\alpha(x) =\kappa(x)$, where $\tau$ is the unit
tangential vector, $\alpha(x)$ denotes the boundary drag coefficient and $\kappa(x)$ is the curvature.

\vskip .1in
In the case when $\mu=0$, \eqref{a1} formally reduces to the inviscid lake equations,
\begin{equation}\label{Inviscid}
\left\{
\begin{array}{l}
\p_t u^{0} +u^{0}\cdot \nabla u^{0}+\nabla
p^0=0,\\[2mm]
\nabla\cdot(b\,u^{0})=0,
\end{array}
\right.
\end{equation}
but the corresponding boundary condition is
\begin{equation}\label{bc}
bu^0\cdot n=0\quad  \hbox{\ on\ } \p\Omega.
\end{equation}

The viscous lake equations \eqref{a1} have been derived to model the evolution
of the vertically averaged horizontal components of the 3D velocity to the
incompressible viscous fluid confined to a shallow basin with a varying
bottom topography (see \cite{CHL1,CHL2,L1}) while the invisvid lake equations
\eqref{Inviscid} describe the evolution of similar physical
quantities governed by the Euler equations (see \cite{G,L3}). Physically $b=b(x)$
denotes the depth of the basin. Our intention here is to deal with the situation
when $b=b(x)$ is degenerate, namely that
$$
b(x) >0 \quad\mbox{for $x \in \Omega$ \,\,and}\quad b(x)=0 \quad\mbox{for $x \in \p\Omega$}.
$$
As in \cite{B4}, we write $\p\Omega$ as the zero level set of a smooth function. That is,
\begin{equation}\label{b-assumption}
b(x)=\varphi(x)^a, \quad \Omega=\{\varphi>0\} \quad \mbox{and}\quad
\p\Omega = \{\varphi=0\},
\end{equation}
where $a>0$ and $\varphi\in C^2(\overline\Omega)$.

\vskip .1in Our goal here is to understand the vanishing viscosity
limit of solutions to the IBVP (\ref{a1})-(\ref{a3}) when the
initial vorticity $\omega_0= b^{-1} \nabla\times u_0 \in
L^q(\Omega)$ for some $q$ satisfying $2<q\le\infty$. To deal with the
vanishing viscosity limit problem, we first establish the global
existence of solutions to the viscous IBVP (\ref{a1})-(\ref{a3})
with $\omega_0\in L^q(\Omega)$ for $2<q\le\infty$.   For the inviscid
IBVP (\ref{Inviscid}),(\ref{bc}) and (\ref{a3}), there is an
adequate theory on the existence and uniqueness of weak solutions.
For the general case $\omega_0 \in L^q(\Omega)$ with
$2<q\le\infty$, a global weak
solution to (\ref{Inviscid}),(\ref{bc}) and (\ref{a3}) in the
distributional sense is obtained in \cite{JN,L3} for nondegenerate
$b(x)$, namely
\begin{equation} \label{nond}
0<b_1 \le b(x) \le b_2 \quad\mbox{for all $x\in \Omega$}.
\end{equation}
When $b(x)$ is degenerate, the global weak solution can be obtained by replacing
$b(x)$ by $b(x)+\epsilon$ for small $\epsilon>0$, applying the result for the
nondegerate case in \cite{JN} and taking the limit as $\epsilon\to 0$. The weak solutions
of (\ref{Inviscid}),(\ref{bc}) and (\ref{a3}) are in the distribution sense and their
uniqueness is unknown if we just have $\omega_0 \in L^q(\Omega)$ with $2<q<\infty$.
If $\omega_0\in L^\infty(\Omega)$, \cite{B4} established the global existence and uniqueness
of weak solutions in the class $\omega\in L^\infty(\Omega\times [0,T])$ for any $T>0$.
With these existence and uniqueness results at
our disposal, we are able to establish two vanishing viscosity limit results. The first
one is the strong convergence
$$
u^\mu \to u^0\quad \mbox{in $L^r(0,T;W^{\alpha,r^{\prime}}(\Omega))$} \quad \mbox{as $\mu\to 0$},
$$
where $u^\mu$ and $u^0$ refer to the aforementioned solutions of
(\ref{a1})-(\ref{a3}) and of (\ref{Inviscid}),(\ref{bc}) and (\ref{a3})
associated with $\omega_0\in L^q$, respectively, and the indices $r$ and $\alpha$ will be
specified later. When $\omega_0\in L^\infty$, an explicit
rate of convergence can be obtained. More precisely, we have
$$
\|\sqrt{b}(u^\mu-u^0)(t)\|_{L^2}^2\le C\,M^{2(1-e^{-\tilde C
t})} \left(\|\sqrt{b}(u^\mu-u^0)(0)\|_{L^2}^2+\mu t\right)^{e^{-\tilde Ct}}.
$$
Precise statements of these results will be given in the following section.

\vskip .1in To put our results in proper context, we briefly
summarize some recent work on the viscous and inviscid lake
equations. When $b=1$, \eqref{a1} and \eqref{Inviscid} become the
classical Navier-Stokes and Euler equations, respectively. There is
a large literature on the inviscid limit of the Navier-stokes
equations with the Navier boundary conditions (see, e.g.,
\cite{Bar,BC,C1,IS,LF,X1}). If $b$ is not a constant but nondegenerate,
namely $b$ satisfies \eqref{nond}, the global existence and
uniqueness of strong solutions to the IBVP (\ref{a1})-(\ref{a3}) is
obtained in \cite{L1} while the global weak solutions to the IBVP
(\ref{Inviscid}),(\ref{bc}) and (\ref{a3}) has been studied by D.
Levermore, M. Oliver and E. Titi in \cite{L3} and \cite{L4}. The vanishing
viscosity limit of (\ref{a1})-(\ref{a3}) in the case when $b$ is
nondegenerate was investigated by Jiu and Niu (\cite{JN}).  They
proved that the solution of (\ref{a1})-(\ref{a3}) with any initial
vorticity in $L^p$ ($1<p \le \infty)$ converges to a weak solution
of (\ref{Inviscid}),(\ref{bc}) and (\ref{a3}). In another recent
work \cite{JN2}, Jiu and Niu studied the viscous boundary layer problem
for \eqref{a1} with Navier boundary conditions.

\vskip .1in
We remark that the vanishing viscosity limit problem for
the case when $b$ is degenerate is more difficult than the nondegenerate case. A key tool
employed here is an elliptic type estimate for degenerate equations (see \cite{B4}
and Lemma \ref{B4} below). This estimate allows us to bound the $W^{1,q}$-norm of $u^\mu$
and $u^0$ uniformly with respect to
the degenerate $b(x)$.  Other techniques involved such as the Yudovich approach will be
unfolded in the subsequent sections.

\vskip .1in
The rest of this paper is divided into three sections. The second section states the
main results and provides tools to be used in the subsequent sections. The third section
establishes the existence and uniqueness of solutions to the IBVP (\ref{a1})-(\ref{a3})
while the last section presents the inviscid limit results.

\vskip .3in
\section{Main Results and Preparations}
\label{prep}
\setcounter{equation}{0}

This section provides the precise statements of the main results
and list some of the tools to be used in the proofs of these theorems.

\vskip .1in
One of the main theorems asserts the global existence and uniqueness
of solutions to the viscous IBVP (\ref{a1})-(\ref{a3}). This theorem
involves the vorticity formulation. If $u^\mu$ solves the IBVP (\ref{a1})-(\ref{a3}), then it can be
verified (see \cite{JN}) that $\omega^{\mu}=b^{-1}\nabla\times
u^{\mu}$ solves the following IBVP for the vorticity equation
\begin{equation}\label{PV}
\left\{
\begin{array}{l}
\p_t \omega^{\mu}+u^{\mu}\cdot\nabla\omega^{\mu}-\mu \Delta
\omega^{\mu}+3\mu b^{-1}\nabla b\cdot\nabla \omega^{\mu} =\mu
G(u^{\mu},\nabla u^{\mu}), \\[2mm]
b \omega^{\mu}=0,\ \ \ \ x\in \p\Omega,\\[2mm]
b \omega^{\mu}(\cdot,0)=b \omega_0,\ \ \ x\in\Omega.
\end{array}
\right.
\end{equation}
where $G(u^{\mu},\nabla u^{\mu})$ involves only the linear terms of the first
derivatives of $u^\mu$, and is given by
\begin{align}
& G = (b^{-1}\Delta b + |\nabla \ln b|^2)\omega^\mu
+ b^{-1}\nabla\times ((\nabla u^\mu\cdot)\ln b)
 \nonumber \\
& \quad + \, b^{-1}\nabla\times (\nabla\ln b(u^\mu\cdot\nabla(\ln b))). \label{Gform}
\end{align}

\begin{thm}\label{vis}
Consider the IBVP (\ref{a1})-(\ref{a3}) with $b=b(x)$ being given
by \eqref{b-assumption} for  $a\geq 2$. Assume $\sqrt{b} u_0\in L^2(\Omega)$ and
$\omega_0= b^{-1} \nabla\times u_0 \in L^q(\Omega)$ for some $q$ satisfying $2<q<\infty$.
Then  (\ref{a1})-(\ref{a3}) has a unique solution which satisfies
\begin{equation*}
\label{a10}\left\{
\begin{array}{l}
\d{d}{dt}\disp\int_{\Omega}\phi \cdot u^{\mu}bdx
+2\mu\disp\int_{\Omega}Du^{\mu}:D\phi \ bdx-\mu\disp\int_{\Omega}divu^{\mu} div\phi\ bdx\\[4mm]
\ \ \ \ \ \ \ \ +\disp\int_{\Omega}u^{\mu}\cdot\nabla u^{\mu}\cdot
\phi\ bdx+2\mu\disp\int_{\p\Omega}\kappa u^{\mu}\cdot \phi b dS
=0,\\[4mm]
b u^{\mu}\cdot n =0,\ \ \ x\in \p\Omega, \\[4mm]
u^{\mu}(x,0)=u_0,\ \ \ x\in\Omega
\end{array}
\right.
\end{equation*}
for any $\phi\in W^{1,\frac{q}{q-1}}(\Omega)$ with $\phi\cdot n=0$ on $\p\Omega$.

\vskip .1in
In addition, $\omega^{\mu}= b^{-1} \nabla\times u^{\mu}$ is well-defined, and
satisfies \eqref{PV} in the distribution sense. Furthermore, for any $T>0$,
$b^{\frac{1}{p}}\omega^{\mu}\in C([0,T]; L^q(\Omega))$ and
\begin{align}\label{c4-1}
&\|\sqrt{b}u^{\mu}\|_{L^{\infty}(0,T;L^2)}+\|
b^{\frac{1}{q}}\omega^{\mu}\|_{ L^{\infty}(0,T;L^q)}\le C,\\[3mm]
&\|u^{\mu}\|_{W^{1,q}}\leq C,\label{c4-10}
\end{align}
where $C$ is a constant depending on $a$, $q$, $T$,
$\|\varphi\|_{C^2(\overline{\Omega})}$ and the initial norms $\|\sqrt{b} u_0\|_{L^2}$ and $\|\omega_0\|_{L^q}$ only.
\end{thm}

\vskip .1in
Since $\Omega$ is a bounded domain,  $\omega_0\in L^\infty(\Omega)$
can be treated as a special case of Theorem \ref{vis}.
\begin{cor} \label{viscor}
Consider the IBVP (\ref{a1})-(\ref{a3}) with $b=b(x)$ being given
by \eqref{b-assumption} for  $a\geq 2$. Assume $\sqrt{b} u_0\in L^2(\Omega)$ and
$\omega_0\in L^\infty(\Omega)$. Then (\ref{a1})-(\ref{a3}) has a
unique solution $u^\mu$ obeys \eqref{c4-1} and \eqref{c4-10} for any $2<q<\infty$.
\end{cor}

It is not clear if the vorticity $\omega^{\mu}$ is in $L^\infty(\Omega)$. The approach
of taking the limit of $\|\omega\|_{L^q}$ as $q\to \infty$ would not work since the
bound for $\|\omega\|_{L^q}$ grows with respect to $q$ very quickly (see
the bound in Lemma \ref{lem-2}).

\vskip .1in
Two other main results are the following theorems on inviscid limits.
The first one is a strong convergence result without an explicit rate. In the
following theorem $u^0$ denotes a weak solution of the inviscid IBVP
(\ref{Inviscid}),(\ref{bc}) and (\ref{a3}) in the distributional sense.
As we explained in the introduction, such weak solutions exist for all time.
For the case when $\omega_0\in L^\infty(\Omega)$, the existence and uniqueness
of weak solutions was obtained by D. Bresch and G. Metivier \cite{B4}.

\begin{lem}\label{good}
Consider the inviscid IBVP (\ref{Inviscid}),(\ref{bc}) and (\ref{a3}) with $b=b(x)$ being given
by \eqref{b-assumption} for  $a\geq 2$. Assume $\sqrt{b} u_0\in L^2(\Omega)$ and
$\omega_0\in L^\infty(\Omega)$. Then (\ref{Inviscid}),(\ref{bc}) and (\ref{a3}) has a
unique solution $u^0$ satisfies, for any $2<p<\infty$ and any $T>0$,
$$
u^0 \in C([0,T]; W^{1,p}), \quad \omega^0 \in C([0,T];
L^p)\cap
L^\infty([0,T]\times\Omega)
$$
and
\begin{equation*}\label{Est-5-30}
\sup_{p\ge 3}\frac
1p\left(\int_\Omega |\nabla u^0|^p dx\right)^{\frac1p} < \infty.
\end{equation*}
\end{lem}

\vskip .1in
We now state our first vanishing viscosity limit result.
\begin{thm}\label{main}
Let $b(x)=\varphi^a$ be given as in \eqref{b-assumption} with
$a\geq 2$. Assume $\sqrt {b} u_0\in L^2(\Omega)$ and $\omega_0\in
L^q(\Omega)$ for some $2<q \le\infty$. Let $u^{\mu}$ be the unique
solution established in Theorem \ref{vis}. Let $\omega^{\mu} =b^{-1} \nabla\times u^{\mu}$.
Then, for any $1<r<\infty$ satisfying $1<1/r+2/q <3/2$,
$$
u^{\mu}\longrightarrow u^0 \ \
\hbox{in}\ \ L^r(0,T;W^{\alpha,r^{\prime}}(\Omega)),
$$
where $r'$ is the conjugate index of $r$, $1/r+1/r'=1$ and
$\alpha\in (0,1)$ satisfies $1/r' < 1/q- (1-\alpha)/2$.
Moreover, $u^0$ is the weak solution
to \eqref{Inviscid} and \eqref{bc}, satisfying, in the case when $2<q<\infty$
$$
\sqrt {b} u^0\in
L^2(\Omega), \quad \omega^0\in L^\infty([0,T], L^q(\Omega))
$$
and, if $q=\infty$, $\omega^0\in L^\infty([0,T], L^{\tilde q}(\Omega))$
for any $1\le\tilde q<\infty$.
\end{thm}

\vskip .1in
We remark that, when $\omega_0\in L^\infty(\Omega)$,  the weak solution $u^{0}$
in Theorem \ref{main} coincides with the unique weak solution in Lemma \ref{good}.

\begin{cor}\label{cor-1}
If $\omega_0\in L^\infty(\Omega)$, the weak solution $u^{0}$
in Theorem \ref{main} coincides with the unique weak solution
in Lemma \ref{good}.
\end{cor}

\vskip .1in
When $\omega_0\in L^\infty(\Omega)$, we
obtain an explicit convergence rate.

\begin{thm}\label{5-thm}
Let $b(x)=\varphi^a$ be given as in \eqref{b-assumption} with
$a\geq 2$. Assume $\sqrt {b} u_0\in L^2(\Omega)$ and $\omega_0\in
L^\infty(\Omega)$. Let $u^{\mu}$ be the unique solution established
in Theorem \ref{vis} and let $u^0$ be the unique weak solution of the
IBVP (\ref{Inviscid}),(\ref{bc}) and (\ref{a3}).
Then, for any $T>0$ and $t\le T$,
$$
\|\sqrt{b}(u^\mu-u^0)(t)\|_{L^2}^2\le C\,M^{2(1-e^{-\tilde C
t})} \left(\|\sqrt{b}(u^\mu-u^0)(0)\|_{L^2}^2+\mu t\right)^{e^{-\tilde Ct}},
$$
where $C$, $\tilde C$ and $M$ are constants depending on $a$, $T$, $\|\varphi\|_{C^2(\overline{\Omega})}$
and the norms $\|\sqrt{b} u_0\|_{L^2}$ and $\|\omega_0\|_{L^\infty}$ only. Especially,
if $\|\sqrt{b}(u^\mu-u^0)(0)\|_{L^2} \to 0$, then $\|\sqrt{b}(u^\mu-u^0)(t)\|_{L^2} \to 0$ with
an explicit rate, as $\mu\to 0$.
\end{thm}

\vskip .1in
We now list some of the tools to be used in the proofs of the theorems
stated above.  The first one is an estimate for solutions of degenerate
elliptic equations. This estimate was obtained in  \cite[Theorem 2.3]{B4}.

\vskip .1in
\begin{lem}\label{B4}
Let $\Omega\subset \mathbb{R}^d$ be a simply connected bounded domain with
a smooth boundary and let $b=b(x)$ be given by \eqref{b-assumption}. Consider
\begin{equation*}\label{deg-1}
\nabla\cdot(bv)=0, \quad \nabla\times v=f \quad \mbox{in $\Omega$} \quad \mbox{and} \quad
(bv)\cdot n =0 \quad \mbox{on $\p\Omega$}.
\end{equation*}
If, for $2<p<\infty$,
\begin{eqnarray*}\label{deg-2}
bv\in L^2(\Omega) \quad \mbox{and}\quad f\in L^p(\Omega),
\end{eqnarray*}
then
\begin{equation*}\label{deg-3}
v\in C^{1-\frac dp}(\overline{\Omega}), \quad \nabla v\in L^p(\Omega), \quad  v\cdot
n|_{\p\Omega}=0
\end{equation*}
and, for a constant $C_p$ depending on $p$ only,
\begin{equation*}\label{deg-3+}
\|v\|_{C^{1-\frac dp}}\le C_p\,(\|f\|_{L^p}+\|b v\|_{L^2}).
\end{equation*}
Especially,
\begin{equation}\label{2.11}
\|v\|_{L^p} \le C \|v\|_{L^\infty} \le
C_p\,(\|f\|_{L^p}+\|bv\|_{L^2}).
\end{equation}
In addition,  for any $p_0>2$ and $p_0<p<\infty$, there is a constant C depending
on $p_0$ only such that
\begin{align}\label{deg-5}
\|\nabla v\|_{L^p}\leq Cp\,(\|f\|_{L^p}+\|bv\|_{L^2}).
\end{align}
\end{lem}

\begin{rem}
The estimates in Lemma \ref{B4} bound the $W^{1,p}$-norm of $v$
uniformly with respect to $b$. The estimates in \eqref{2.11} and
 \eqref{deg-5} actually hold for $p=2$, namely the $H^1$-norm of $v$
is bounded by $C (\|f\|_{L^2}+\|bv\|_{L^2})$.
\end{rem}

\vskip .1in
The following lemma reformulates the Navier friction condition in
terms of vorticity (see, e.g., \cite{LF}).

\begin{lem}\label{Nav}
Suppose $v\in H^2(\Omega)$ with $v\cdot n =0$ on $\p\Omega.$  Then,
\begin{equation*}\label{a6}
D(v)n\cdot \tau = - \kappa(v\cdot\tau)  + \d{1}{2}\nabla\times v\ \ \ \hbox{on}\ \
\p\Omega,
\end{equation*}
where $\tau$ denotes the unit tangent vector and $\kappa$ the curvature of
$\p\Omega.$  In particular,  if $\nabla\times v=0$ on $\partial\Omega$, then
\begin{equation*}\label{a6+}
D(v)n\cdot \tau=-\kappa(v\cdot\tau)\ \ \ \hbox{on}\ \ \p\Omega.
\end{equation*}
\end{lem}

\vskip .1in
We will also need the following  Osgood type inequality(see, e.g., \cite{Chemin}).

\begin{lem}\label{gronwall-2}
Let $\alpha(t)>0$ be a locally integrable function. Assume
$\omega(t)\ge 0$ satisfies
$$
\int_0^\infty \frac{1}{\omega(r)} dr = \infty.
$$
Suppose that $\rho(t)>0$ satisfies
$$
\rho(t)\le a +\int_{t_0}^t \alpha(s)\omega(\rho(s)) ds
$$
for some constant $a\ge 0$. Then if $a=0$, then $\rho\equiv 0$; if
$a>0$, then
$$
-\Omega(\rho(t))+\Omega(a)\le \int_{t_0}^t \alpha(\tau) d\tau,
$$
where
$$
\Omega(x)=\int_x^1 \frac{dr}{\omega(r)}.
$$
\end{lem}

\vskip .3in
\section{Global solutions of the viscous equations}

This section is devoted to the proof of Theorem \ref{vis}. For this purpose,
we first establish several {\it a priori} estimates including a global
$L^2$-bound for the velocity, a global $L^q$-bound for the vorticity and
a global $L^2_tH^1_x$ bound for the velocity.

\vskip .1in
We start with the $L^2$-bound for the velocity.

\begin{lem}($L^2$-Estimate)\label{lem-1}
Suppose that the assumptions of Theorem \ref{vis} hold and  let $u^\mu$
be a smooth solution of \eqref{a1}. Then, for any $T>0$,
\begin{equation}\label{lem-1+}
\|\sqrt{b}u^\mu\|^2_{L^\infty((0,T);L^2(\Omega))}+\int_0^T\int_{\p\Omega}\kappa
|u^{\mu}\cdot \tau|^2 b dS\le \|\sqrt{b}u_0\|^2_{L^2(\Omega)},
\end{equation}
where $\kappa\ge 0$ is the curvature of $\p\Omega$.
\end{lem}
\begin{proof}
We take the inner product of the first equation of \eqref{a1} with $bu^\mu$ and
integrate by parts. Due to the divergence free condition $\nabla\cdot (b u^\mu)=0$,
the contribution from the nonlinear term and the pressure term is zero. The inner
product with the dissipative term is
\begin{align*}
\begin{split}
& \mu \int_{\Omega} u^\mu \cdot \nabla\cdot (2b Du^\mu - b \nabla\cdot u^\mu I) \,dx \\
& = -\mu \int_{\p\Omega} ( 2 u^\mu \cdot D u^\mu n -  (u^\mu\cdot n) \nabla\cdot u^\mu) b dS \\
&\quad + 2 \mu \int_{\Omega} \nabla u^\mu : D u^\mu b dx- \mu \int_{\Omega}
(\nabla\cdot u^\mu)^2 b dx.
\end{split}
\end{align*}
Writing $u^\mu = (u^\mu\cdot n) n +  (u^\mu\cdot \tau) \tau$, applying the
boundary condition in (\ref{a2}) and the basic identity $\nabla u^\mu : D u^\mu
=D u^\mu: D u^\mu$, and invoking Lemma \ref{Nav}, namely $D(u^\mu)n\cdot
\tau=-\kappa(u^\mu\cdot\tau) \ \hbox{on}\ \p\Omega$, we have
\begin{eqnarray}\label{Lem-eqn1}
&&\d{d}{dt}\disp\int_{\Omega}|u^{\mu}|^2bdx
+2\mu\disp\int_{\Omega}Du^{\mu}:Du^\mu \ bdx-\mu\disp\int_{\Omega}(\nabla\cdot u^{\mu})^2\ b dx\nonumber\\[4mm]
&&+2\mu\disp\int_{\p\Omega}\kappa |u^{\mu}\cdot \tau|^2 b dS =0,
\end{eqnarray}
where $\kappa$ is the curvature of $\partial\Omega$ which is
nonnegative by assumption. Since
$$
2Du^{\mu}:Du^\mu-(\nabla\cdot
u^\mu)^2=(\p_1u^\mu_2+\p_2u^\mu_1)^2+(\p_1u^\mu_1-\p_2u^\mu_2)^2 \ge 0,
$$
\eqref{lem-1+} then follows from \eqref{Lem-eqn1}. The proof of the lemma is then finished.
\end{proof}

\vskip .1in
For the vorticity $\omega^\mu= b^{-1} \nabla\times u^\mu$, we have the following estimate.

\begin{lem}(Estimate of Vorticity)\label{lem-2}
Suppose that the assumptions of Theorem \ref{vis} hold and  let $u^\mu$
be a smooth solution of \eqref{a1}. Let $\omega^\mu= b^{-1} \nabla\times u^\mu$.
Then, for any $T>0$,
\begin{align}\label{L2-6}
\|(b)^{\frac{1}{q}}\omega^{\mu}\|_{L^{\infty}(0,T;L^q)}^q
\leq (\|\sqrt{b} u_0\|_{L^2}^q+\|\omega_0\|_{L^q}^q) e^{\mu
(Cq)^{q+1}T},
\end{align}
where $C$ is a constant depending on $a$, $q$, $T$ and
$\|\varphi\|_{C^2(\overline{\Omega})}$.
\end{lem}

\begin{proof}
As stated in Section \ref{prep}, $\omega^\mu$ satisfies (\ref{PV}).
Taking the inner product of $|\omega^{\mu}|^{q-2} \omega^{\mu}b$ with
the first equation of (\ref{PV}), integrating by parts and using the zero
boundary condition for $b \omega^\mu$, we have
\begin{align}\label{c9}
&\frac{1}{q}\d{d}{dt}\|b^{\frac{1}{q}}\omega^{\mu}\|_{L^q}^q
+\frac{4(q-1)}{q^2}\mu\disp\int_{\Omega}|\nabla
(\omega^{\mu})^{\frac{q}{2}}|^2b dx
\nonumber\\
&\le \mu\left|\disp\int_{\Omega}|G(u^{\mu}, \nabla
u^{\mu})||\omega^{\mu}|^{q-2}\omega^{\mu}b dx\right|
+ 4\mu \left|\disp\int_{\Omega}\nabla b\cdot\nabla \omega^{\mu} |\omega^{\mu}|^{q-2}\omega^{\mu} dx\right|.\nonumber
\end{align}
To bound the first term, we first notice from \eqref{Gform} that
$$
\|b G(u^{\mu}, \nabla
u^{\mu})\|_{L^q} \le \|u^\mu\|_{W^{1,q}}.
$$
It then follows from H\"{o}lder's inequality that
$$
\mu\left|\disp\int_{\Omega}|G(u^{\mu}, \nabla
u^{\mu})||\omega^{\mu}|^{q-2}\omega^{\mu}b dx\right| \le C \mu \|u^\mu\|_{W^{1,q}}\|\omega^{\mu}\|_{L^q}^{q-1}.
$$
To bound the last term, we recall that $b=\varphi^a$ with $\varphi\in C^2(\overline\Omega)$
and $\varphi\ge 0$. Therefore, for $a\ge 2$,
\begin{equation}\label{bnb}
|\nabla b|^2 = |a\varphi^{a-1}\nabla \varphi|^2 \le C \varphi^{2a-2} \le C \varphi^a =C b.
\end{equation}
Thus, by H\"{o}lder's and Young's inequalities,
\begin{equation*}
\int_{\Omega} |\nabla b\cdot\nabla \omega^{\mu}| |\omega^{\mu}|^{q-2}\omega^{\mu}
dx \le \frac{\mu}{q}\int |\nabla (\omega^\mu)^{\frac q2}|^2b
dx+\frac{C \mu}{q}\|\omega^\mu\|_{L^q}^q,
\end{equation*}
where $C$ is independent of $q$. Therefore, we obtain
\begin{eqnarray}\label{L2-2}
&&\frac{1}{q}\d{d}{dt}\|b^{\frac{1}{q}}\omega^{\mu}\|_{L^q}^q
+\frac{3q-4}{q^2}\mu\disp\int_{\Omega}|\nabla
(\omega^{\mu})^{\frac{q}{2}}|^2b dx
\nonumber\\[3mm]
&&\leq \frac{C \mu}{q} \|\omega^{\mu}\|_{L^q}^q+
\mu\|u^\mu\|_{W^{1,q}}\|\omega^{\mu}\|_{L^q}^{q-1}.\nonumber
\end{eqnarray}
By the estimates in Lemma \ref{B4},
$$
\|\omega^{\mu}\|_{L^q}\leq \|\nabla u^{\mu}\|_{L^q}\leq
Cq\,(\|b\,\omega^{\mu}\|_{L^q}+\|b u^{\mu}\|_{L^2}).
$$
Thus,
\begin{eqnarray}\label{L2-3}
&&\d{d}{dt}\|b^{\frac{1}{q}}\omega^{\mu}\|_{L^q}^q
+\frac{3q-4}{q}\mu\disp\int_{\Omega}|\nabla
(\omega^{\mu})^{\frac{q}{2}}|^2b dx
\nonumber\\[3mm]
&&\le \mu
(Cq)^q(\|b\,\omega^\mu\|_{L^q}+\|b u^\mu\|_{L^2})^q\nonumber\\[3mm]
&&\le\mu (Cq)^{q+1} (\|b\omega^\mu\|_{L^q}^q+\|bu^\mu\|_{L^2}^q). \nonumber
\end{eqnarray}
Noticing that $\|b\omega^\mu\|_{L^q} \le \|b^{1/q}\omega^\mu\|_{L^q}$ and applying
Lemma \ref{lem-1}, we have
\begin{equation*}\label{c11}
\|b^{\frac{1}{q}}\omega^{\mu}\|_{L^{\infty}(0,T;L^q)}^q \leq (\|\sqrt{b} u_0\|_{L^2}^q+\|\omega_0\|_{L^q}^q)e^{\mu (Cq)^{q+1}T},
\end{equation*}
which is \eqref{L2-6}. The proof of the lemma is complete.
\end{proof}

\vskip .1in
The following lemma provides a bound for $\|\sqrt{b} \nabla u\|_{L^2(\Omega\times [0,T])}$.
In addition, its proof is also useful in proving Theorem \ref{5-thm}.
\begin{lem}\label{lem-1new}
Suppose that the assumptions of Theorem \ref{vis} hold and  let $u^\mu$
be a smooth solution of \eqref{a1}.
Then, for any $T>0$,
\begin{align}
\begin{split}\label{lem-11}
& \|\sqrt{b}u^\mu\|^2_{L^\infty((0,T);L^2(\Omega))}+\mu\int_0^T\|\sqrt{b}\nabla
u^\mu(t)\|^2_{L^2(\Omega)} dt \\
& \qquad +\mu\int_0^T\int_{\p\Omega}\kappa |u^{\mu}\cdot \tau|^2 b dS dt \le
C(\|\sqrt{b}u_0\|^2_{L^2(\Omega)}+\|\omega_0\|^2_{L^2(\Omega)}).
\end{split}
\end{align}
\end{lem}

\begin{proof}
Substituting the identity
$$2Du^{\mu}:Du^\mu-(\nabla\cdot
u^\mu)^2=|\nabla
 u^\mu|^2+2(\p_1 u^\mu_2\p_2u_1^\mu-\p_1u^\mu_1\p_2u^\mu_2)
$$
into \eqref{Lem-eqn1}, we obtain
\begin{eqnarray}\label{Lem-eqn2}
&&\d{d}{dt}\disp\int_{\Omega}|u^{\mu}|^2bdx
+\mu\disp\int_{\Omega}|\nabla
 u^\mu|^2b dx-2\mu\int_{\Omega}(\p_1u^\mu_1\p_2u^\mu_2-\p_1 u^\mu_2\p_2u_1^\mu) b dx\nonumber\\[4mm]
&&+2\mu\disp\int_{\p\Omega}\kappa |u^{\mu}\cdot \tau|^2 b dS =0.
\end{eqnarray}
It is easy to check that
\begin{eqnarray*}
 J &\equiv&  2\int_{\Omega}(\p_1u^\mu_1\p_2u^\mu_2-\p_1 u^\mu_2\p_2u_1^\mu)b dx\\[3mm]
&=&\int_{\Omega} \nabla\cdot (u_1^\mu\p_2u^\mu_2-u^\mu_2\p_2u^\mu_1,
u^\mu_2\p_1u_1^\mu-u_1^\mu\p_1u_2^\mu)b dx\\[3mm]
&=&\int_{\Omega} \nabla\cdot [(u_1^\mu\p_2u^\mu_2-u^\mu_2\p_2u^\mu_1,
u^\mu_2\p_1u_1^\mu-u_1^\mu\p_1u_2^\mu)b] dx\\[3mm]
&&-\int_{\Omega}(u_1^\mu\p_2u^\mu_2-u^\mu_2\p_2u^\mu_1,
u^\mu_2\p_1u_1^\mu-u_1^\mu\p_1u_2^\mu)\cdot \nabla b dx.
\end{eqnarray*}
Writing $$(u_1^\mu\p_2u^\mu_2-u^\mu_2\p_2u^\mu_1,
u^\mu_2\p_1u_1^\mu-u_1^\mu\p_1u_2^\mu)
$$ $$
= u_1^\mu (\p_2u^\mu_2, -\p_1u_2^\mu) - u^\mu_2 (\p_2u^\mu_1, -\p_1u_1^\mu)
$$
and applying the divergence theorem, we have
\begin{eqnarray*}
&& J=\int_{\p\Omega} (u_1^\mu\tau\cdot\nabla
u_2^\mu-u_2^\mu\tau\cdot\nabla  u_1^\mu)b
dS\\[3mm]
&&\qquad -\int_{\Omega}(u_1^\mu\p_2u^\mu_2-u^\mu_2\p_2u^\mu_1,
u^\mu_2\p_1u_1^\mu-u_1^\mu\p_1u_2^\mu)\cdot \nabla b dx.
\end{eqnarray*}
Since $b u^\mu\cdot n =0$ on $\p\Omega$, we have $b u^\mu=(b u^\mu\cdot \tau)\tau$
on $\p\Omega$. Writing $u_1^\mu\tau\cdot\nabla u_2^\mu-u_2^\mu\tau\cdot\nabla  u_1^\mu
=-\tau\cdot\nabla u^\mu\cdot (u_2^\mu, -u_1^\mu)$, we find
\begin{eqnarray*}
&&J =-\int_{\p\Omega} (\tau\cdot\nabla u^\mu\cdot n)(u^\mu\cdot\tau) b
dS\\[3mm]
&&\qquad -\int_{\Omega}(u_1^\mu\p_2u^\mu_2-u^\mu_2\p_2u^\mu_1,
u^\mu_2\p_1u_1^\mu-u_1^\mu\p_1u_2^\mu)\cdot \nabla b
dx.
\end{eqnarray*}
By Lemma \ref{Nav},
\begin{eqnarray*}
&&J=\int_{\p\Omega} \kappa|u^\mu\cdot \tau|^2b
dS\\
&&\qquad -\int_{\Omega}(u_1^\mu\p_2u^\mu_2-u^\mu_2\p_2u^\mu_1,
u^\mu_2\p_1u_1^\mu-u_1^\mu\p_1u_2^\mu)\cdot \nabla b dx.
\end{eqnarray*}
Then \eqref{Lem-eqn2} becomes
\begin{eqnarray}\label{Lem-eqn3}
&&\d{d}{dt}\disp\int_{\Omega}|u^{\mu}|^2bdx
+2\mu\disp\int_{\Omega}|\nabla
 u^\mu|^2b dx\nonumber+\mu\disp\int_{\p\Omega}\kappa |u^{\mu}\cdot \tau|^2 b dS \\[3mm]
&&=-\mu\int_{\Omega}(u_1^\mu\p_2u^\mu_2-u^\mu_2\p_2u^\mu_1,
u^\mu_2\p_1u_1^\mu-u_1^\mu\p_1u_2^\mu)\cdot \nabla b dx.
\end{eqnarray}
Applying H\"{o}lder's inequality and using (\ref{bnb}), we have
\begin{eqnarray}\label{Lem-eqn5}
&&\mu \left|\int_{\Omega}(u_1^\mu\p_2u^\mu_2-u^\mu_2\p_2u^\mu_1,
u^\mu_2\p_1u_1^\mu-u_1^\mu\p_1u_2^\mu)\cdot \nabla b dx\right|\nonumber\\[3mm]
&&\le \frac12 \mu\disp\int_{\Omega}|\nabla
 u^\mu|^2b dx+ C\,\mu\|u^\mu\|^2_{L^2}\nonumber\\[3mm]
 &&\le\frac12 \mu\disp\int_{\Omega}|\nabla
 u^\mu|^2b dx+C\mu(\|b^{1/2}\omega^\mu\|^2_{L^2}+\|b^{1/2} u^\mu\|^2_{L^2}).
 \end{eqnarray}
Combining \eqref{Lem-eqn3} with \eqref{Lem-eqn5}, we obtain
\begin{eqnarray}\label{Lem-eqn6}
&&\d{d}{dt}\disp\int_{\Omega}|u^{\mu}|^2bdx
+\mu\disp\int_{\Omega}|\nabla
 u^\mu|^2b dx\nonumber\\[3mm]
&&\quad \le C\mu(\|b^{1/2}\omega^\mu\|^2_{L^2}+\|b^{1/2} u^\mu\|^2_{L^2}).\nonumber
\end{eqnarray}
Applying \eqref{L2-6} and the Gronwall inequality, we obtain \eqref{lem-11} and thus finish the proof of this lemma.
\end{proof}

\vskip .1in
We are now ready to prove Theorem \ref{vis}.
\begin{proof}[Proof of Theorem \ref{vis}]
Let $\epsilon>0$ be a small parameter. We construct the approximate solutions
$(u^{\epsilon,\mu},\omega^{\epsilon,\mu})$ to the nondegenerate
viscous lake equations with $b^{\epsilon}=b+\epsilon$, namely
\begin{equation}\label{c3}
\left\{
\begin{array}{l}
\p_t u^{\epsilon, \mu} +u^{\epsilon,\mu}\cdot \nabla u^{\epsilon,\mu}
\\[2mm]
\qquad -\mu
(b^{\epsilon})^{-1}\nabla\cdot(2b^{\epsilon}D(u^{\epsilon,\mu})-b^{\epsilon }\nabla\cdot u^{\epsilon,\mu} I)
+\nabla p^{\epsilon,\mu}=0,\\[2mm]
\nabla\cdot(b^{\epsilon}u^{\mu})=0,\\[2mm]
b^{\epsilon} u^{\epsilon,\mu}\cdot n=0, \ b^{\epsilon}\omega^{\mu}=0\quad \mbox{on $\p\Omega$},\\[2mm]
u^{\epsilon, \mu}(x,t)\mid_{t=0}=u_0.
\end{array}
\right.
\end{equation}
Since $b^\epsilon$ is nondegenerate, the global existence and uniqueness of such solutions can be obtained
by a similar approach as in \cite{JN}.  Moreover, $u^{\epsilon,\mu}$ satisfies \eqref{c3} in the
sense of distribution
\begin{align}\label{c6}
&\d{d}{dt}\disp\int_{\Omega}\phi \cdot u^{\epsilon,\mu}b^{\epsilon}dx
+2\mu\disp\int_{\Omega}Du^{\epsilon,\mu}:D\phi \ b^{\epsilon}dx\nonumber\\
&\qquad-\mu\disp\int_{\Omega}divu^{\epsilon,\mu} div\phi\ b^{\epsilon}dx
+\disp\int_{\Omega}u^{\epsilon,\mu}\cdot\nabla u^{\epsilon,\mu}\cdot \phi\
b^{\epsilon}dx\nonumber\\
&\qquad +\mu\disp\int_{\p\Omega}\kappa
(u^{\epsilon,\mu}\cdot\phi)b^{\epsilon}dS =0
\end{align}
for $\phi\in W^{1,\frac{p}{p-1}}(\Omega)$ with $\phi\cdot n=0$ on $\p\Omega$.
Thanks to Lemma \ref{lem-1} and Lemma \ref{lem-2}, we deduce the
uniform estimates, for any $T>0$,
\begin{align}\label{c4}
\|\sqrt{b^{\epsilon}}u^{\epsilon,\mu}\|_{L^{\infty}(0,T;L^2)}+\|
(b^{\epsilon})^{\frac{1}{q}}\omega^{\epsilon,\mu}\|_{
L^{\infty}(0,T;L^q)}\le C.
\end{align}
By the estimates in Lemma \ref{B4},
\begin{align}
\begin{split}\label{c5}
& \|u^{\epsilon,\mu}\|_{W^{1,q}}\leq
C(\|\sqrt{b^{\epsilon}}u^{\epsilon,\mu}\|_{L^{\infty}(0,T;L^2)} \\
& \qquad\qquad \qquad +\|(b^{\epsilon})^{\frac{1}{q}}\omega^{\epsilon,\mu}\|_{
L^{\infty}(0,T;L^q)})\le C.
\end{split}
\end{align}
In these inequalities $C$'s are constants depending on $T$ and $q$
but not on $\epsilon$ or $\mu$. Furthermore, using \eqref{c6}, we can
prove that $\p_t u^{\epsilon,\mu}$ is uniformly bounded in
$L^\infty((0,T);H^{-s}_{loc}(\Omega))$ for some $s>2$. Thus
\eqref{c4} and \eqref{c5} yield the compactness of
$\sqrt{b^{\epsilon}}u^{\epsilon,\mu}$ in
$L^2(0,T;L^2_{loc}(\Omega))$ by Aubin-Lions Lemma. This allows to
pass to the limit $\epsilon\to 0$ in \eqref{c6} to get the existence
of weak solutions of \eqref{a1}-\eqref{a3}. Moreover, the solution
$u^\mu, \omega^\mu$ satisfy the estimates of \eqref{c4} and
\eqref{c5}. Using similar estimates of \eqref{Lem-eqn6} and
\eqref{c11}, we can prove uniqueness of the weak solutions and we
omit further details. The proof of the theorem is now finished.
\end{proof}

\vskip .3in
\section{Vanishing Viscosity limits}

This section proves Theorem \ref{main},
and Theorem \ref{5-thm}, the vanishing viscosity limit results. In addition, a proof of Corollary \ref{cor-1} is also provided at the end of this section.

\vskip .1in
\begin{proof}[Proof of Theorem \ref{main}]  According to Theorem \ref{vis} and its proof
given in the previous section, the unique solution $u^\mu$ of the IBVP (\ref{a1})-(\ref{a3})
satisfies
$$
\sqrt{b}u^{\mu}\in C(0,T;L^2)\cap L^2(0,T;H^1(\Omega)),
$$
$$
u^{\mu}\in L^{\infty}(0,T;W^{1,q}(\Omega)), \quad b^{\frac{1}{q}}\omega^{\mu}\in L^{\infty}(0,T;L^q(\Omega))
$$
and, for any test function $\phi\in C([0,T); W^{1, \frac{q}{q-1}})$ with $\phi\cdot n=0$ on $\p\Omega$,
\begin{eqnarray*}\label{c33}
&&\disp\int_{\Omega}\phi u^{\mu}bdx
+2\mu\disp\int_0^T\disp\int_{\Omega}Du^{\mu}:D\phi bdx
+\mu\disp\int_0^T\disp\int_{\Omega}\nabla\cdot u^{\mu} div\phi bdx\nonumber\\[3mm]
&&\qquad +\disp\int_0^T\disp\int_{\Omega}u^{\mu}\cdot\nabla u^{\mu}\cdot
\phi bdx+\mu\disp\int_0^T\disp\int_{\p\Omega}\kappa
(u^{\mu}\cdot\phi)bdS\nonumber\\[3mm]
&&\qquad =\disp\int_{\Omega}u_0 \phi(0,\cdot)bdx.
\end{eqnarray*}
Then we can take a subsequence, denoted by ${u^{\mu_k}}$, such that
$$
\begin{array}{l}
u^{\mu_k}\rightharpoonup u^0\ \    \hbox{\ \ in\ \ }w\ast-
L^{\infty}(0,T; W^{1,q}(\Omega))\cap L^{\infty}(0,T; L^2(\Omega)),\\[2mm]
\omega^{\mu_k}\rightharpoonup \omega^0,\ \   \hbox{\ \ in\ \ }
w\ast- L^{\infty}(0,T; L^q(\Omega)),
\end{array}
$$
as $k\longrightarrow \infty.$ Therefore, for any $1<r<\infty$ satisfying $1<1/r+2/q <3/2$
and $\alpha\in (0,1)$ satisfying $1/r' < 1/q- (1-\alpha)/2$,
$$
u^{\mu}\longrightarrow u^0 \ \ \hbox{in}\ \
L^r(0,T;W^{\alpha,r^{\prime}}(\Omega)),
$$
where $r'$ is the conjugate index of $r$, $1/r+1/r'=1$.

\vskip .1in
In addition, the limiting function $u^0$ satisfies the weak form
of the inviscid lake equations, that is,
\begin{eqnarray*}
\disp\int_{\Omega}\phi u^0 bdx
+\disp\int_0^T\disp\int_{\Omega}u^0\cdot\nabla u^0\cdot \phi
bdx=\disp\int_{\Omega}u_0 \phi(0,\cdot)bdx.
\end{eqnarray*}
This completes the proof.
\end{proof}

\vskip .1in
We now turn to the proof of Theorem \ref{5-thm}.

\begin{proof}[Proof of Theorem \ref{5-thm}] The differences $v=u^{\mu}-u^0$ and
$p=p^{\mu}-p^0$ formally satisfy
\begin{align}\label{d1}
\left\{
\begin{array}{l}
\partial_t v+v\cdot \nabla u^0+u^\mu\cdot \nabla v\\
\qquad\qquad\qquad -\mu b^{-1}
\nabla\cdot(2bDu^{\mu}-b\nabla\cdot u^{\mu})
+\nabla p=0,\\
\nabla\cdot(bv)=0,
\end{array}
\right.
\end{align}
with the boundary condition $b v\cdot n=0$. Taking the inner product of \eqref{d1} with $bv$,
integrating by parts and applying the boundary conditions, we obtain
\begin{eqnarray}\label{d3}
&&\frac{1}{2}\frac{d}{dt}\|\sqrt{b}v\|_{L^2(\Omega)}^2+
\int_{\Omega}v\cdot \nabla u^0\cdot vb dx
+2\mu\int_{\partial\Omega} \kappa |v|^2 b dS\nonumber\\
&&\quad +2\mu\int_{\Omega}D(v):D(v)b dx
-\mu\int_{\Omega}(\nabla\cdot v)^2b dx\nonumber\\
&&=-2\mu \int_{\partial\Omega}\kappa u^0\cdot vb
dS-2\mu\int_{\Omega}D(u^0): D(v) bdx
\nonumber\\
&& \quad +\mu\int_{\Omega} (\nabla\cdot u^0) (\nabla\cdot v) bdx.
\end{eqnarray}
We remark that \eqref{d3} can be obtained rigorously by using the
weak form of the equations. We then combine the terms
$$
2\mu\int_{\Omega}D(v):D(v)b dx
-\mu\int_{\Omega}(\nabla\cdot v)^2b dx
$$
and bound them as in the proof of Lemma \ref{lem-1new}. More
explicitly, as calculations in lemma \ref{lem-11}, we  write
\begin{eqnarray*}
&&2\mu\int_{\Omega}D(v):D(v)b dx
-\mu\int_{\Omega}(\nabla\cdot v)^2b dx \\
&&= \mu\disp\int_{\Omega}|\nabla
 v|^2b dx-2\mu\int_{\Omega}(\p_1v_1\p_2v_2-\p_1 v_2\p_2v) b dx\\
 &&=\mu\disp\int_{\Omega}|\nabla
 v|^2b dx-\mu\int_{\p\Omega} \kappa|u^\mu\cdot \tau|^2b
dS\\
&&\qquad +\mu\int_{\Omega}(u_1^\mu\p_2u^\mu_2-u^\mu_2\p_2u^\mu_1,
u^\mu_2\p_1u_1^\mu-u_1^\mu\p_1u_2^\mu)\cdot \nabla b dx,
\end{eqnarray*}
and then bound the last term above as in \eqref{Lem-eqn5}, namely
\begin{eqnarray*}
&& \mu\left|\int_{\Omega}(u_1^\mu\p_2u^\mu_2-u^\mu_2\p_2u^\mu_1,
u^\mu_2\p_1u_1^\mu-u_1^\mu\p_1u_2^\mu)\cdot \nabla b dx\right| \\
&&\quad \le \frac12 \mu\disp\int_{\Omega}|\nabla
 v|^2b dx + C\mu(\|b^{1/2} u^\mu\|^2_{L^2} + \|b^{1/2}\omega^\mu\|^2_{L^2})\\
&&\quad \le \frac12 \mu\disp\int_{\Omega}|\nabla  v|^2b dx
+ C\mu(\|b^{1/2} u^0\|^2_{L^2} + \|b^{1/2} u^\mu\|^2_{L^2})\\
&&\qquad\qquad \qquad   + \, C\mu(\|b^{1/2}\omega^0\|^2_{L^2} + \|b^{1/2}\omega^\mu\|^2_{L^2})\\
&&\quad \le \frac12 \mu\disp\int_{\Omega}|\nabla  v|^2b dx + C \mu,
\end{eqnarray*}
where $C$'s depend on the initial norms $\|b^\frac12 u_0\|_{L^2}$
and $\|\omega_0\|_{L^\infty}$ only. Applying H\"{o}lder's inequality
and Lemma \ref{B4}, we have, for any $T>0$ and $t\le T$,
\begin{eqnarray}\label{Sec5-1}
&&\frac{1}{2}\frac{d}{dt}\|\sqrt{b}v\|_{L^2(\Omega)}^2
 + \frac12\mu\int_{\Omega}|\nabla v|^2 b dx
+\mu \int_{\partial\Omega}\kappa\,|v|^2 b dS\nonumber\\
&&\quad \le C\mu + \left|\int_{\Omega}v\cdot \nabla u^0\cdot vb
dx\right|\nonumber\\
&&\qquad +2\mu \left(\int_{\partial\Omega}\kappa |u^0|^2 b
dS\right)^{1/2} \left(\int_{\partial\Omega}\kappa
|v|^2 b dS\right)^{1/2}\nonumber\\
&& \qquad +2\mu\|\nabla
u^0\|_{L^2(\Omega)}\left(\int_{\Omega}|D(v)|^2
bdx\right)^\frac12\nonumber\\
&& \qquad +\mu\|\nabla\cdot
u^0\|_{L^2(\Omega)}\left(\int_{\Omega}(\nabla\cdot v)^2b
dx\right)^\frac12.
\end{eqnarray}
Applying the bounds $\|b^{1/2}u^0\|_{L^2} \le C$  for $C$
independent of $\mu$ and  by Lemma \ref{B4},
\begin{eqnarray*}
\|\nabla u^0\|_{L^2(\Omega)} &\le&  C\,\|\nabla
u^0\|_{L^3(\Omega)}\nonumber\\
&\le&
C(\|b \omega^0\|_{L^3(\Omega)}+\|b u^0\|_{L^2(\Omega)})\\
&\le& C,
\end{eqnarray*}
where $C$'s depend on the initial norms $\|b^\frac12 u_0\|_{L^2}$
and $\|\omega_0\|_{L^\infty}$ only, we have from \eqref{Sec5-1} that
\begin{eqnarray}\label{Sec5-3}
\frac{d}{dt}\|\sqrt{b}v\|_{L^2(\Omega)}^2 +
\frac{\mu}{2}\int_{\Omega}|\nabla v|^2 b dx\le 2
\left|\int_{\Omega}v\cdot \nabla u^0\cdot vb dx \right| + C\mu,
\end{eqnarray}
where $C$ is independent of $\mu$. Since $\nabla u^0$ is not known to be bounded in
$L^\infty$, we follow the Yudovich approach to deal with the nonlinear term
(see, e.g., \cite{Yod} and \cite{B4}). For this purpose, we set
\begin{eqnarray*}
&&L:=\sup_{p\ge 3}\frac
1p\left(\int_\Omega |\nabla u^0|^p dx\right)^{\frac1p},\\
&&M:=\|u^0\|_{L^\infty}+\|u^\mu\|_{L^\infty}.
\end{eqnarray*}
By Lemma \ref{good}, $L<\infty$ and by Lemma \ref{B4}, $M<\infty$.
Now, for $\delta>0$,  let
$$
\Gamma_{\mu, \delta}(t)=\|\sqrt{b} v \|_{L^2(\Omega)}^2 + \delta.
$$
Applying H\"{o}lder's inequality to the nonlinear term in \eqref{Sec5-3}, we have,
for any $p\ge 3$,
\begin{eqnarray}\label{Sec5-5}
&&\frac{d}{dt} \Gamma_{\mu, \delta}(t) \le
pLM^{\frac2p} \Gamma_{\mu, \delta}(t)^{1-\frac1p} + C\mu.
\end{eqnarray}
Optimizing the bound on the right of \eqref{Sec5-5} with
respect to $p\ge 3$ yields
\begin{eqnarray*}
&&\frac{d}{dt}  \Gamma_{\mu, \delta} (t) \le Ce(\ln M^2-\ln
\Gamma_{\mu, \delta}(t)) \Gamma_{\mu, \delta} (t)+C\mu.
\end{eqnarray*}
Integrating in time leads to
$$
 \Gamma_{\mu, \delta} (t)\le  \Gamma_{\mu, \delta} (0)+  C \mu
t+Ce\int_0^t\rho( \Gamma_{\mu, \delta} (\tau)) d\tau,
$$
where $\rho(x)=x(\ln M^2-\ln x)$. Let
\begin{eqnarray*}
\Omega(x)&=&\int_x^1 \frac{dy}{\rho(y)}=\int_x^1 \frac{dy}{y(\ln
M^2-\ln y)}\nonumber\\
&=&\ln(\ln M^2-\ln x)-\ln\ln M^2.
\end{eqnarray*}
Applying Lemma \ref{gronwall-2}, we get
$$
-\Omega(\Gamma_{\mu, \delta}(t))+\Omega(\Gamma_{\mu, \delta}(0) + C\mu t)\le \tilde C t,
$$
where $C$ and $\tilde C$ are constants independent of $\mu$. Therefore,
$$
-\ln(\ln M^2-\ln \Gamma_{\mu, \delta}(t)) + \ln(\ln M^2-\ln(\Gamma_{\mu, \delta}(0)+C\mu
t))\le \tilde C t.
$$
That is,
$$
\Gamma_{\mu, \delta}(t)\le M^{2(1-e^{-\tilde C t})} (\Gamma_{\mu, \delta}(0)+ \mu
t)^{e^{-\tilde Ct}}.
$$
Letting $\delta\to 0$, we obtain
$$
\|\sqrt{b}(u^\mu-u^0)(t)\|_{L^2}^2\le C\,M^{2(1-e^{-\tilde C
t})} \left(\|\sqrt{b}(u^\mu-u^0)(0)\|_{L^2}^2+\mu t\right)^{e^{-\tilde Ct}}.
$$
This completes the proof of Theorem \ref{5-thm}.
\end{proof}

\vskip .1in
We finally prove Corollary \ref{cor-1}.

\begin{proof}[Proof of Corollary
\ref{cor-1}] Let $\omega_0\in L^\infty(\Omega)$ and let
$u^0_1$ and $u^0_2$ be weak solutions given by  Lemma \ref{good} and
Theorem \ref{main}, respectively. Then, the difference
$$
\bar u^0=u^0_1-u_2^0
$$
satisfies the energy inequality
\begin{eqnarray*}
\disp \frac{d}{dt} \int_{\Omega}|\bar u^0|^2 bdx \le \disp
2\int_0^T\disp\int_{\Omega}|\bar u|^2|\nabla u_1^0| b dx.
\end{eqnarray*}
A Yudovich type argument as in the previous proof would lead to $\bar u^0=0$, or $u_1^0=u_2^0$. We have thus completed the proof.
\end{proof}

\vskip .4in
\section*{Acknowledgements}
Jiu's research was  partially supported by NSFC grants No. 10871133 \& No. 10771177.
Niu's research was partially supported by NSFC grant No. 10871133. Wu's research was partially
supported by NSF grant DMS 0907913 and the AT\&T foundation at Oklahoma State University.

\end{document}